\documentclass[11pt]{article}

\usepackage[utf8]{inputenc}
\usepackage[T1]{fontenc}
\usepackage[english]{babel}

\usepackage[margin=1in]{geometry}
\usepackage{setspace}
\usepackage{relsize}

\usepackage{amsmath,amsthm,amssymb,mathtools}
\usepackage{mathrsfs}

\usepackage{csquotes}

\usepackage{hyperref}
\hypersetup{
    colorlinks=true,
    linkcolor=blue,
    citecolor=blue,
    urlcolor=blue,
    pdftitle={Persistent Cost of Lipschitz Maps}
}

\input xy
\xyoption{all}

\newtheorem{thm}{Theorem}[section]
\newtheorem{cor}[thm]{Corollary}
\newtheorem{defn}[thm]{Definition}
\newtheorem{ex}[thm]{Example}
\newtheorem{rem}[thm]{Remark}
\newtheorem{lem}[thm]{Lemma}

\newenvironment{pf}[1][Proof]{\noindent\textbf{#1.} }{ \hfill$\square$\medskip}

\DeclareMathOperator{\VR}{VR}
\DeclareMathOperator{\dist}{dist}

\DeclareMathOperator{\Coker}{Coker}
\DeclareMathOperator{\Ker}{Ker}
\DeclareMathOperator{\Img}{Im}

\renewcommand{\H}{\operatorname{H}}

\newcommand{\bu}{_{\bullet}}
\newcommand{\bue}{_{\bullet+\epsilon}}

\newcommand{\RR}{\mathbb{R}}

\newcommand{\C}{\mathcal{C}}

\newcommand{\argmin}{\operatorname{argmin}}

\usepackage{graphicx}
\usepackage{tikz}
\usetikzlibrary{decorations.pathreplacing}
\usetikzlibrary{math,arrows.meta,shapes, positioning}
\tikzset{    
    mypoint/.style={
        circle,
        draw,
        inner sep=.3mm
        },  
    whitepoint/.style={
        fill=white, 
        mypoint
        },  
    blackpoint/.style={
        fill=black, 
        mypoint
        },  
    textnode/.style={
        text height=2.5ex, 
        text depth=1ex
        },  
    }

\newcounter{height}
\title{Persistent Cost of Lipschitz Maps}
\author{
  Francisco J. Gozzi\footnotemark[1] \and
  Manuela A. Cerdeiro\footnotemark[2] \footnotemark[3] \and
  Pablo E. Riera\footnotemark[2] \footnotemark[3]
}
\begin{document}

\maketitle

\begin{abstract}
A $1$-Lipschitz map between compact metric spaces $f\colon X\to Y$ induces a homomorphism of persistence modules on degree-$d$ Vietoris--Rips persistent homology. We define the persistent cost of $f$ from this induced homomorphism by quantifying the persistence carried by its kernel and cokernel modules.

We prove that the persistent cost controls the interleaving distance between the degree-$d$ Vietoris--Rips persistent homology modules of $X$ and $Y$. Moreover, we obtain an explicit upper bound for the persistent cost in purely metric terms. Finally, we give a self-contained proof of the stability of the persistent cost introducing a Gromov-Hausdorff type distance for maps between compact metric spaces.
\end{abstract}

\footnotetext[1]{Centro de Matemática, Computação e Cognição da Universidade Federal do ABC, SP, Brasil.}
\footnotetext[2]{Universidad de Buenos Aires. Facultad de Ciencias Exactas y Naturales. Departamento de Computación. Buenos Aires, Argentina.}
\footnotetext[3]{CONICET-Universidad de Buenos Aires. Instituto de Ciencias de la Compuatción (ICC). Buenos Aires, Argentina.}

\section*{Introduction}

Persistent homology is a fundamental tool in the emerging field of topological data analysis (TDA) that extracts a vast amount of metric information by encoding homological features across scales.
The richness of this output coupled with its explainability has yielded significant insights across a spectrum of disciplines, including biology and time series analysis, among others, c.f. \cite{otter2017roadmap, perea2015sliding, gardner2022toroidal}, and been integrated into learning pipelines, c.f. \cite{hensel2021survey, hofer2017deep}.

Given a metric space \( X \), the Vietoris--Rips filtration \( (\VR_s(X))_{s \geq 0} \) is constructed by including simplices whose vertex set has diameter at most \( s \geq 0 \). As the parameter increases, the complex evolves from a discrete set to the full power set of \( X \), encoding geometric information at multiple scales. Applying the homology functor in a fixed dimension \( d \) yields a persistence module: a sequence of vector spaces connected by linear maps that track the birth, persistence and death of homological features. 

While the TDA community has primarily focused on analyzing individual point clouds, the functorial nature of these constructions naturally leads us to consider maps between datasets.
Indeed, we advocate for a treatment of maps between spaces analogous to that of the spaces themselves, enabling us to address the geometry of functions from a persistent homology perspective.
In this context we shall consider the natural case of non-expanding, i.e., 1-Lipschitz, maps $f:X\to Y$ which induce morphisms between their Vietoris–Rips filtrations and hence homomorphisms of persistence modules, 
\[ f_{*}:\H_d(\VR\bu(X))\to \H_d(\VR\bu(Y)).\]

A homomorphism of persistence modules defines kernel, cokernel, and image modules via point-wise linear algebra, along with naturally induced maps.
The \enquote{size} of the modules $\Ker(F)$ and $\Coker(F)$, measured by their interleaving distance to the trivial module $\Delta$, serves as a measure of how far the map $F$ is from a persistent isomorphism. 
This leads us to define the \textit{persistent cost} of a $1$-Lipschitz function as the maximum interleaving distance of kernel and cokernel persistence modules to the diagonal, 
\begin{equation}\label{eq persistent cost}
\begin{array}{rcl}
C(f)   &  = & \max \big\{ d_I(\Ker(f), \Delta )~,~ d_I( \Coker(f), \Delta ) \big\}
\end{array}
\end{equation}

The persistent cost bounds the interleaving distance between the modules associated to $X$ and $Y$, by a technical improvement to the Induced Matching Theorem of U. Bauer and M. Lesnick in \cite{bauer2015induced}, see Theorem \ref{thm lower bound}.

On the other hand, we can explicit a metric upper bound for the persistent cost of a non-expanding map $f:X\to Y$, in terms of its distortion and Hausdorff distance between its image subset and codomain, $f(X)\subseteq Y$. Altogether, we establish the following sequence of inequalities:
 \begin{equation}\label{eq main inequality}
d_{I} \left( \H_d (\VR\bu(X) ), \H_d (\VR\bu(Y) ) \right) ~\leq ~C(f) ~\leq~  \dist(f) + 2\cdot d_H(f(X), Y) 
 \end{equation}
We show that this bounds are sharp. 

From a purely algebraic perspective, as we range over all possible persistence module homomorphisms $F\bu:V\bu\to W\bu$, the infimum of the associated persistent cost attains the interleaving distance between the given modules. However, our interest lies in geometrically defined persistence homomorphisms, i.e. those induced by $1$-Lipschitz maps between the underlying spaces via the homology of their Vietoris-Rips filtrations. In this geometric context, we provide an example of two finite metric spaces that yield identical persistence diagrams while any 1-Lipschitz map is bounded away from zero in its associated persistent cost.

A related issue concerns the stability of kernel and cokernel persistence modules, which was claimed without proof in \cite{cohen2009persistent} as a corollary of the machinery developed in \cite{cohen2005stability}. 
We provide here a self-contained exposition of this fact introducing a Gromov-Hausdorff type distance for maps $f:X\to Y$, $f':X'\to Y'$, between compact metric spaces, that may be of independent interest.


\medskip
The paper is organized as follows. A first section provides a minimal exposition of the algebraic notions that allow us to define the persistent cost and prove the first inequality above, Theorem \ref{thm lower bound}. We then move on to the geometry of spaces and maps, in order to establish the metric upper bound on Theorem \ref{thm main 1 - metric bound} and stability on Theorem \ref{thm main 2 - stability}.

\section{Algebraic Persistence}

The main references for this section are the comprehensive book by S. Oudot in \cite{oudot2015persistence} and the papers on the Induced Matching Theorem \cite{bauer2015induced,bauer2020persistence} by U. Bauer and M. Lesnick.

\subsection{Persistence Modules}

A persistence module is a linear representation of $\RR$, or a given subset $T\subseteq \RR$, regarded as a poset. In other words, it is a functor
\begin{equation}\label{eqn defn pers module}
V: (\mathbb{R}, \leq) \to \mathrm{Vect}_\mathbb{F},
\end{equation}
where $\mathrm{Vect}_\mathbb{F}$ denotes the category of vector spaces over a field $\mathbb{F}$ with linear transformations as homomorphisms. 
Thus, a persistence module is more than just a one-parameter family of vector spaces; as a representation of a poset, it carries \enquote{inner persistence} or \enquote{transition} maps $\varphi_{s<t}: V_s \xrightarrow{~~} V_t$, for each pair $s<t \in \RR$, which satisfy:
\[ \varphi \circ \varphi = \varphi, \]
the last equation being read as $\varphi_{s_2<s_3} \circ \varphi_{s_1<s_2} = \varphi_{s_1<s_3}$.



Additional hypotheses are often imposed on persistence modules to ensure that their representation category has satisfactory structural properties. A possible choice is to restrict to pointwise finite-dimensional (p.f.d.) modules, which admit an interval decomposition due to the work in \cite{crawley2015decomposition}. A broader condition is $q$-tameness, requiring that each internal structure map $\varphi_{s<t}$ has finite rank.

As a basic example, let $(X,d)$ be a metric space. For each scale parameter $r\ge 0$, the Vietoris--Rips complex $\mathrm{VR}(X;r)$ is the abstract simplicial complex whose $k$-simplices are the finite subsets $\{x_0,\dots,x_k\}\subseteq X$ with $d(x_i,x_j)\le r$ for all $i,j$. Considering homology with coefficients in a fixed field $\Bbbk$ over the previous filtration, produces a persistence module which is clearly p.f.d. 
if $X$ is finite. If $X$ is compact, the module need not be p.f.d. in general, but it is $q$-tame, i.e, for each $r$ and each $\epsilon>0$ the space
\begin{equation}\label{eq epsilon approx of V}
V^{\epsilon}_r=\mathrm{im}\,\varphi_{r-\epsilon<r}    
\end{equation}
has finite dimension. 
In particular, $V^{\epsilon}\bu$ is a p.f.d. submodule of $V\bu$. 

Furthermore, there is a natural $\epsilon$-interleaving between a module $V\bu$ and its $\epsilon$-approximation $V^{\epsilon}\bu$, the pair of maps being given by the inclusion $V^{\epsilon}\bu\subseteq V\bu$ and inner persistence from $\varphi_{\bullet<\bullet+\epsilon}: V\bu \to V^{\epsilon}\bue$.

\subsection{Homomorphisms and Interleavings}
\smallskip
\begin{minipage}{0.685\textwidth}
\begin{defn} \label{def homomorphism}
A homomorphism $F: V\bu\to W\bu$ between persistence modules is a natural transformation between $V\bu$ and $W\bu$ regarded as functors.\end{defn}

In other words, a homomorphism is a collection of linear maps \enquote{$F\bu$}, at each scale, commuting with inner maps of each module, so that for $s<t$ we have the diagram to the side.    

\end{minipage}
\begin{minipage}{0.25\textwidth}
\[\xymatrix{ 
V_s \ar[r]^{\varphi^V_{s<t}} \ar[d]_{F_s} & V_{t} \ar[d]_{F_t} \\
W_s \ar[r]_{\varphi^W_{s<t}}  & W_t.
} \]
\end{minipage}


\smallskip
\noindent
\begin{minipage}{0.684\textwidth}
For $\epsilon \in \RR$, a $\epsilon$-degree homomorphism between given modules $V\bu$ and $W\bu$, is a standard persistence module homomorphism when the codomain is shifted by $\epsilon$, namely  $F: V\bu \to W_{\bue}$.  
\end{minipage}
\begin{minipage}{0.314\textwidth}
\[ \xymatrix{ 
V_s \ar[r]^{\varphi^V} \ar[dr]_{F_s} & V_{t} \ar[dr]_{F_t}& \\
&W_{s+\epsilon} \ar[r]_{\varphi^W}  & W_{t+\epsilon}.
} \]
\end{minipage}

\begin{defn}\label{defn interleaving}
 A $\delta$-interleaving between persistence modules $V_{\bu}$ and $W_{\bu}$ is a pair of degree $\delta>0$ homomorphisms
 $F:V\bu \to W\bu$, $G:W\bu \to V\bu$,
 such that the following diagrams commute:
\[
\xymatrix{ V_{\bullet} \ar[rr] \ar[dr]_{F} & & V_{\bullet+ 2\delta} 
&& V_{\bullet + \delta} \ar[dr]^{F} & \\
& W_{\bullet + \delta} \ar[ur]_{G} & 
& W_{\bullet} \ar[rr] \ar[ur]^{G} & & W_{\bullet+ 2\delta} 
}\]
\end{defn}
Whenever it is clear from context we omit indexing homomorphisms and naming inner persistence maps of each module. 

\begin{defn}
The interleaving distance between persistence modules $V\bu,W\bu$ is defined by:
\[d_I(V_{\bullet}, W_{\bullet}) = \inf\{\epsilon \geq 0: \text{there is  an $\epsilon$-interleaving between $V\bu$ and $W\bu$}\} \]
\end{defn}  
Though referred as \enquote{distance}, it is only an extended pseudo-distance, i.e., it may attain the value $+\infty$ and assign zero distance to different modules. 

The interleaving distance agrees with the Bottleneck distance between associated diagrams, in the interval decomposable case where the latter is defined. 
This result is known as the algebraic isometry theorem, see \cite{lesnick2015theory, oudot2015persistence}.  

\subsection{Persistent Cost}

Associated to a homomorphism $F: V\bu \to W\bu$, we can define persistence modules given as the point-wise kernel, cokernel, and image spaces, with persistence maps naturally induced from those of $V\bu$ and $W\bu$.  In particular, we retrieve an exact sequence of persistence modules given by:
\begin{equation} \label{eqn exact sequence ker coker}  
0 \to \Ker(F)\bu \xrightarrow{~~~}  V\bu \xrightarrow{~~F~~} W\bu  \xrightarrow{~~~} \Coker(F)\bu \to 0. \end{equation}

The size of the kernel and cokernel encode how far $F$ is from being a persistence isomorphism, motivating the following.
\begin{defn}
The persistent cost associated to a homomorphism $F:V\bu \to W\bu$ of persistence modules is 
the larger of the interleaving distances of the kernel and the cokernel to the trivial persistence module $\Delta$,  
\[C(F) = \max\left\{ d_I(\operatorname{Ker}(F), \Delta),\ d_I(\operatorname{Coker}(F), \Delta) \right\}.
\]
\end{defn}

We may now state our first contribution. 
\begin{thm}\label{thm lower bound}
Let $F:V\bu\to W\bu$ be a homomorphism between q-tame persistence modules, then its persistent cost bounds the interleaving distance between $V\bu$ and $W\bu$, i.e.,
\[ d_I(V\bu,W\bu)\leq C(F). \]
\end{thm}
\begin{pf}
If the modules are p.f.d. the statement follows as a corollary to the Induced Matching Theorem of U. Bauer and M. Lesnick in \cite{bauer2015induced}. 

Recall the auxiliary definition of $V^{\epsilon}\bu$ 
in \eqref{eq epsilon approx of V} and observe that a persistence homomorphism $F$ corestricts well to:
\[F^{\epsilon}:=F|: V^{\epsilon}\bu\to W^{\epsilon}\bu . \]
In particular from \cite{bauer2015induced} we have:
 \[ d_I( V^{\epsilon}\bu, W^{\epsilon}\bu) \leq C(F^{\epsilon}).  \]
We need only check that:
\[  d_I(\Ker(F), \Ker(F^{\epsilon}))= \epsilon  \quad,\quad   d_I(\Coker(F), \Coker(F^{\epsilon})= \epsilon,\]
where the pair of maps that give the desired interleavings are induced from those in \eqref{eq epsilon approx of V} between each module $V\bu$, $W\bu$ and its corresponding $\epsilon$-approximation
$V^{\epsilon}\bu$, $W^{\epsilon}\bu$.

A straight-forward triangular inequality gives:
 \[ d_I( V\bu, W\bu) \leq 2\epsilon + C(F^{\epsilon})\leq 4 \epsilon + C(F).  \]
Given that $\epsilon>0$ can be set up independently, the theorem follows.
\end{pf}

Notice that this bound is sharp, a simple example given by the null homomorphism from the trivial module into any other.

\section{On maps}
 A general reference for this section is the book by Oudot \cite{oudot2015persistence} with regards to persistence theory. Our metric arguments elaborate on classical metric geometry, for which we refer to \cite{burago2001metricgeometry}.

\subsection{Quasi Lipschitz maps}

Let us consider perturbations of $1$-Lipschitz functions as follows.
\begin{defn}
A function between metric spaces, $h:X\to Y$, is said to be an $\epsilon$-quasi $1$-Lipschitz map if it satisfies the following inequality:
\[ d_Y\big(h(x),h(x')\big)\leq d_X\big(x,x'\big)+\epsilon.\]
\end{defn}

The Vietoris-Rips construction is particularly well suited for our metric endeavors, as the following result shows.

\begin{lem}\label{lemma induced function on homology}
Let \( h : X \to Y \) be an \( \epsilon \)-quasi \( 1 \)-Lipschitz map. Then \( h \) induces an \( \epsilon \)-degree persistence homomorphism: $
h_* : H_d(\operatorname{VR}\bu(X)) \to H_d(\operatorname{VR}_{\bullet+\epsilon}(Y)).$

Moreover, if two \( \epsilon \)-quasi \( 1 \)-Lipschitz maps \( h_1, h_2 : X \to Y \) are within uniform distance \( \delta \geq 0 \), then they induce the same \( (\epsilon + \delta) \)-degree map on homology.
\end{lem}

\begin{proof}
The map \( h \) induces a simplicial map:
\[
\hat{h} : \operatorname{VR}_{\bullet}(X) \to \operatorname{VR}_{\bullet+\epsilon}(Y), \quad [x_0 \dots x_r] \mapsto [h(x_0) \dots h(x_r)],
\]
since a set \( \sigma \subseteq X \) of diameter \( \ell \) maps to a set \( h(\sigma) \subseteq Y \) of diameter at most \( \ell + \epsilon \). This commutes with inclusions and thus gives a homomorphism between filtrations. The homology functor yields the desired persistence homomorphism.

For the second claim, suppose \( d_Y(h_1(x), h_2(x)) \leq \delta \) for all \( x \in X \). For any simplex \( \sigma \in \operatorname{VR}_{\bullet}(X) \) of diameter \( \ell \), consider the union:
\[
h_1(\sigma) \cup h_2(\sigma) \subseteq Y.
\]
For any \( y, y' \in h_1(\sigma) \cup h_2(\sigma) \), we have:
- If both come from the same \( h_i \), then \( d(y, y') \leq \ell + \epsilon \).
- If \( y = h_1(x) \), \( y' = h_2(x') \), then:
\[
d(y, y') \leq d(h_1(x), h_2(x)) + d(h_2(x), h_2(x')) \leq \delta + (\ell + \epsilon).
\]
Thus, the diameter of \( h_1(\sigma) \cup h_2(\sigma) \) is at most \( \ell + \epsilon + \delta \). This implies that \( \hat{h}_1 \) and \( \hat{h}_2 \) are contiguous as maps:
\[
\operatorname{VR}_{\bullet}(X) \to \operatorname{VR}_{\bullet+\epsilon+\delta}(Y),
\]
and therefore induce the same map on homology.
\end{proof}

\begin{rem}
If \( f : X \to Y \) is a \( 1 \)-Lipschitz map, it induces a persistence homomorphism:
\[
f_* : H_d(\operatorname{VR}_{\bullet}(X)) \to H_d(\operatorname{VR}_{\bullet}(Y)).
\]
We may then define the \emph{persistent cost} of \( f \) as:
\[
C(f) := C(f_*) = \max\left\{ d_I(\operatorname{Ker}(f_*), \Delta),\ d_I(\operatorname{Coker}(f_*), \Delta) \right\},
\]
where \( \Delta \) is the trivial persistence module.

As natural examples one should consider the isometric inclusion of a sample withing a larger space $S\subseteq Y$, a clustering assignment $p: X\to X/\sim$ with a convenient metric or, in fact, any map $f:X\to Y$ between finite metric spaces as long as the distance of the target space $Y$ is conformally adjusted to ensure that the 1-Lipschitz condition is satisfied. 
\end{rem}


\begin{defn}
    A correspondence is a relation between two sets, $\C\subseteq X\times Y$, that relates every element of the factors, i.e., such that $\pi_X(\C)=X$, $\pi_Y(\C)=Y$.
\end{defn}
When these spaces $X,Y$, are metric one can define the distortion associated to a correspondence as the supremum of the difference between the distances of related pairs, i.e., 
\begin{equation}\label{eq distortion defn}
dist(\C)= sup \Bigg\{ |d(x,x')-d(y,y')|  ~:~(x,y),(x',y')\in \C \Bigg\}    
\end{equation}
Within this framework we may state the Gromov-Hausdorff classical stability and formulate the proof of our main theorem.  
\begin{defn} 
Given metric spaces $X$, $Y$, the Gromov-Hausdorff distance between them is defined by, 
\[ d_{GH}(X,Y) = \frac{1}{2} \, \underset{\C}{\inf} \{\!~dist(\C)\!~\}, \]
where the infimum is taken over all possible correspondences. 
\end{defn}
The alternative, albeit original, definition is to consider the infimum of the Hausdorff distance between isometric copies of $X,Y$:  
\[ d_{\mathrm{GH}}(X,Y) = \inf_{Z,f,g} d_{\mathrm{H}}^Z(f(X), g(Y)), \]
where the infimum ranges over isometric embeddings $f: X \to Z$, $g: Y \to Z$ into a common ambient space $Z$.

The Gromov-Hausdorff distance is a distance when we consider compact metric spaces up to isometry. 

\begin{cor}  
A correspondence with $\epsilon$-distortion between metric space $X$ and $Y$ induces an \mbox{$\epsilon$-interleaving} of the persistence modules given by the d-homology of their Vietoris-Rips filtrations.
\end{cor}
\begin{pf}
Notice that associated to a correspondence $\C$ between spaces $X$ and $Y$ there is a choice of a function $h:X\to Y$ whose graph is a sub-relation of $\C$.
This map is an $\epsilon$-quasi $1$-Lipschitz map for $\epsilon=dist(\C)$ and, hence, by Lemma \ref{lemma induced function on homology} induces a map on homology as required. 

In turn, we retrieve a function from $Y$ to $X$, which together with the previous gives an interleaving.  
\end{pf}

\begin{thm}[Stability \cite{cohen2005stability}]  \label{thm GH stability}

Given compact spaces $X, Y$, their interleaving distance between the associated d-homology Vietoris-Rips persistence module is bounded by twice the Gromov-Hausdorff distance between the spaces themselves.
\[ d_I \big( \H_d(\VR\bu(X)), \H_d(\VR\bu(Y)) \big) \leq 2\,d_{GH}(X,Y). \]    
\end{thm}
This key result provides a central feature of persistent homology which is its stability, namely, the fact that the considered invariants are continuous on the inputs.

Let us list a few examples of naturally occurring correspondences between compact metric spaces $X,Y$ associated to an $\epsilon$-quasi $1$-Lipschitz function $f:X\to Y$.

\begin{ex}\label{example correspondences 1}
The co-restriction of a given map between metric spaces, $f|:X\to f(X)$, being automatically surjective is already a correspondence. Its distortion can be interpreted as the absolute error between the distances of inputs and outputs under $f$,  
    \[dist(f|)=sup_{x,x'}\big| d(f(x),f(x')) - d(x,x') \big| .\]
\end{ex}
 
\begin{ex}\label{example correspondences 2}
An inclusion of compact metric spaces $A\subseteq Y$ can be enlarged to a correspondence $C_{\iota}$ by adding all the pairs $(a,y)\in A\times Y$ for $a = argmin~ d(.,y)$. Its distortion is (sharply) bounded by twice the Hausdorff distance between $A$ and $Y$, i.e., 
\[dist(\C_{\iota})= max\big\{|d(a,a') - d(y,y')| : (a,y),(a',y') \in \C \big\}\leq 2\cdot d_{H}(A,Y).\]  
\end{ex}

\begin{ex}\label{example correspondences after function}
Given that a composition of correspondences is a correspondence, the previous two examples show that a any function $f:X\to Y$ between compact metric spaces can be enlarged as a relation to induce a correspondence between $X$ and $Y$:
\[\C_f=\{(x,y) \in X\times Y~:~ x=\argmin~d(f(.),y) \}.\]

Moreover, as a composition its distortion is bounded as follows:
\[ dist(\C_f) \leq dist(f)+2 \cdot d_{H}(f(X),Y).\]
\end{ex}

\medskip

\subsection{A metric upper bound} 

We now have the necessary framework to address the main topic of this work which is the persistent cost associated to a non-expanding map between compact metric spaces. 
We shall provide natural upper bounds and stability guarantees.

\begin{thm}\label{thm main 1 - metric bound} Let $f:X\to Y$ be a 1-Lipschitz map between compact metric spaces, its persistent cost $C(f)$ on the level of the Vietoris-Rips filtration $d$-dimensional homology satisfies the following inequality:
\[ C(f)~\leq~ dist(f)~+~2\cdot d_H(f(X), Y)    , \] 
in terms of the distortion of the original map and the Hausdorff distance between its image set and codomain.
\end{thm}
Notice that this is a uniform bound holding for all dimensions in which the homology is computed, alike the classic stability result, Theorem \ref{thm GH stability}.

\begin{pf}
Our argument consists on the construction of a natural interleaving between $X\bu$ and $Y\bu$ associated to a $1$-Lipschitz map $f$ that in fact factors through $f_*$ and allows us to control kernel and cokernel modules.
This is the interleaving associated to the geometric correspondence defined in Example \ref{example correspondences after function}.

We have two maps subordinated to this correspondence: the original \(f \colon X \to Y\) and, in the reverse direction, a map \(g \colon Y \to X\) defined as \(g = f^{-1} \circ N\), where \(N \colon Y \to f(X)\) picks a nearest point in the image \(f(X)\) to a given \(y \in Y\), and \(f^{-1}\) denotes a choice of preimage under \(f\). We have a commuting a diagram:
\[ \xymatrix{ 
\H_d(\VR\bu(X)) \ar[rr]^{}  \ar[d]_{f_*} \ar[dr]_{}&  & \H_d(\VR_{\bullet+2\delta+4h}(X)) \\  
\H_d(\VR\bu(Y)) \ar[r]_{} & \H_d(\VR_{\bullet+\delta+2h}(Y))\ar[ur]_{g_*}  & 
} \]

\smallskip
Together with the following diagram, we have a $(\delta+2h)$-interleaving on homology.
\[\xymatrix{ 
& \H_d(\VR_{\bullet+\delta+2h}(X)) \ar[dr]^{~\varphi\circ f_*} &  \\  
\H_d(\VR\bu(Y))  \ar[rr]_{} \ar[ur]^{g_*~} & & \H_d(\VR_{\bullet+2\delta+4h}(Y))   
} \] 
\medskip

In order to consider (co)-kernel modules, we  recall the exact sequence \eqref{eqn exact sequence ker coker} associated to the persistence homomorphism $f_*$:
\[  0 \to \Ker\bu \xrightarrow{~~~}  \H_d( \VR\bu (X)) \xrightarrow{~~f_*~~} \H_d( \VR\bu (Y))  \xrightarrow{~~~} \Coker\bu \to 0. \] 
 
One should observe that the restriction of the diagonal maps coming from the interleaving co-restrict well to the kernel.

\[ \xymatrix{ 
\Ker_{\bullet} \ar[rr]^{} \ar[d]^{} &  & \Ker_{\bullet+2\delta+4h}\ar[d]^{} \\  
\H_d( \VR\bu (X)) \ar[rr]^{}  \ar[d]_{f_*} \ar[dr]_{}&  & \H_d( \VR_{\bullet+2\delta+4h} (X)) \\  
\H_d( \VR\bu (Y)) \ar[r]_{} & \H_d( \VR_{\bullet+\delta+2h}(Y)) \ar[ur]_{}  & 
} \]\vspace{0.7cm}

In doing so, we obtain an interleaving between the trivial module $\Delta$ and $\Ker$, using the fact that we have factored through the map $f_*$.
\
Moving on, we show that $\Coker$ is $\delta+2h$ interleaved with the trivial module, by proving that its inner persistence map $\varphi^{\Coker}_{\delta+2h}$ is trivial. 
This follows from the previous interleaving established between $\H_d( \VR\bu (X))$  and $\H_d( \VR\bu (Y))$, and the fact that this inner map is given by projection onto the cokernel, i.e., 
$\varphi^{\Coker}=\pi \circ \varphi^{Y}$.
It follows that:
\begin{equation}
\begin{array}{rl}
 \varphi^{\Coker}_{2\delta+4h} & = \pi \circ \varphi^Y_{2\delta+4h},\\[6pt]
  & = \pi \circ  (\varphi^Y\circ f_*\circ g_*),\\[6pt]
 &= \varphi^Y\circ (\pi \circ f_*)\circ g_*,\\[6pt]
 &= 0.\\
\end{array}
\end{equation}
Where we have employed the fact that projection onto the cokernel is a persistence module homomorphism in itself. 
\end{pf}

\subsection{Stability} \label{subsection stability}

Stability of kernel, cokernel and image persistence modules associated to a map has been claimed in \cite{cohen2009persistent} without proof but as a standard corollary to the techniques developed in \cite{cohen2005stability}.
We give a self-contained exposition of stability for the case of Vietoris-Rips filtrations.

\bigskip
\noindent\begin{minipage}{0.72\textwidth}
We define a Gromov-Hausdorff type distance for functions $f_1:X_1\to Y_1$, $f_2:X_2\to Y_2$, between possibly different (compact) spaces. 
In order to do so, we consider auxiliary $\epsilon$-isometries $S,S',T,T',$ between domains and codomains, respectively,  
 as in the diagram on the side.
\end{minipage}
\begin{minipage}{0.24\textwidth}
\centering
\[
\xymatrix{ 
X_1 \ar[r]^{f_1} \ar@/_/[d]_{S} & Y_1 \ar@/_/[d]_{T}\\
X_2 \ar[r]_{f_2} \ar@/_/[u]_{S'} & Y_2 \ar@/_/[u]_{T'}
}\]
\end{minipage}
\medskip

Composing the functions $f_i$, $i=1,2$, with suitable $\epsilon$-isometries we retrieve maps with a common domain and codomain and compute their uniform distance: 
$ d_{\infty}\big(f_2\circ S, T \circ f_1 \big)$  and $d_{\infty}\big(f_1\circ S',T' \circ f_2\big)$.

\begin{defn} 
For a pair of functions $f_1:X_1\to Y_1$, $f_2:X_2\to Y_2$ between compact metric spaces, consider:
\[
\begin{array}{ll}
d_{GH}(f_1,f_2) =  \underset{S,S',T,T'}{\inf} \max\Big\{&   
d_{\infty}\big(f_2\circ S, T \circ f_1 \big),~~ d_{\infty}\big(f_1\circ S',T' \circ f_2\big), \\
&\qquad\qquad  dist(S),~dist(S'),~dist(T),~dist(T')      \Big\} .
\end{array} \]
The infimum is taken over all possible almost isometries $S,S',T,T'$.
\end{defn}

\begin{rem} This is a pseudo-distance between functions $f_i:X_i\to Y_i$ defined on compact spaces. Moreover,
\begin{enumerate}

    \item it refines the Gromov-Hausdorff distance between respective domains and codomains, 
    \[  2\cdot \max \{   d_{GH}(X_1,X_2)~,~ d_{GH}(Y_1,Y_2)\} \leq d_{GH}(f_1,f_2) ;\]

    \item it is invariant under isometries, i.e., 
 \[ d_{GH}(f_1,f_2)= d_{GH}(f_1\circ S~,~T \circ f_2).\]
for true isometries $S:X_1\to X_2$, $T:Y_1\to Y_2$, between the respective spaces.
  
\item it induces a true distance on the classes of functions up to composition with isometries. 
\end{enumerate}
\end{rem}

To the best of our knowledge this is a new distance, although it is fairly common subject on the literature to build on the Gromov-Hausdorff formalism in order to define new distances, c.f., \cite{khezeli2023unified}.

\begin{lem}\label{lem bounding (co)kernel, images}
A pair of $1$-Lipschitz functions $f_i:X_i\to Y_i$, $i=1,2$, between compact metric spaces gives persistent homomorphisms $f_i:=(f_i)_*: \H_d(\VR\bu(X_i)) \to \H_d(\VR\bu(Y_i))$, 
whose induced kernel, image and cokernel persistence modules are 
$2\epsilon$-interleaved, provided the functions lay at distance $d_{GH}(f_1,f_2)\leq\epsilon$, so that: 
\[  \max\big\{d_I(\Ker(f_1),\Ker(f_2)),~ d_I(\Img(f_1),\Img(f_1)),~d_I(\Coker(f_1),\Coker(f_2)) ~\big\}\leq~ 2\cdot d_{GH}(f_1,f_2)\]
\end{lem}
\begin{pf}
Fix $\epsilon$-isometries $S,S',T,T',$  associated to the given correspondences: $C_X,C_X^t,C_Y,C_Y^t$, respectively, and such that $ d_{\infty}\big(f_2\circ S, T \circ f_1 \big), ~d_{\infty}\big(f_1\circ S',T' \circ f_2\big) \leq \epsilon$. 

Lemma \ref{lemma induced function on homology} and the fact that $f_2\circ S$ and $T\circ f_1$ are both $\epsilon$-quasi Lipschitz maps at $\epsilon$ uniform distance from each other yield the following commutative diagram:

\begin{equation}\label{diagram epsilon commutativity}
\xymatrix{
\H_d(\VR\bu(X_1)) \ar[d]^{(f_1)_*} \ar[r]_{S_*} & \H_d(\VR_{\bullet+\epsilon}(X_2)) \ar[r]_{\varphi^{X_2}_{\epsilon}}  & \H_d(\VR_{\bullet+2\epsilon}(X_2)) \ar[d]^{(f_2)_*} \\
\H_d(\VR\bu(Y_1)) \ar[r]_{T_*} & \H_d(\VR_{\bullet+\epsilon}(Y_2)) \ar[r]_{\varphi^{X_1}_{\epsilon}} & \H_d(\VR_{\bullet+2\epsilon}(Y_2)).
}\end{equation}

In particular, $\varphi^{X_2}_{\epsilon}\circ S_*$ sends $\Ker((f_1)_*)$ into $\Ker((f_2)_*)$, since
$(f_2)_*$ and the corresponding $\varphi_{\epsilon}$ commute. 
Analogously, $\varphi^{X_1}_{\epsilon}\circ S'_*$ sends $\Ker((f_2)_*)$ into $\Ker((f_1)_*)$. 
Hence, the pair $\varphi^{X_2}_{\epsilon}\circ S_*$ and $\varphi^{X_1}_{\epsilon}\circ (S')_*$ co-restricts well to give a  $2\epsilon$-interleaving between the kernels, since the original pair $S_*,S'_*$ was clearly an $\epsilon$-interleaving between the modules $\H_d(\VR\bu(X_i))$, $i=1,2$, corresponding to the domains. 

Finally, a technicality. One should consider the previous argument for a sequence of $\epsilon'>\epsilon$ converging to the true value $\epsilon=d_{GH}(f_1,f_2)$. Since the interleaving distance is also defined as an infimum, the lemma follows.  

\smallskip
We omit the proof for the image and cokernel persistence modules, which make use of $\epsilon$-isometries $T,T',$ and the inner persistence maps on the codomains.
\end{pf}

As a corollary to the previous proof we have the following.

\begin{cor}\label{cor bounding (co)kernel, images for almost lipschiztz}
A pair of $\delta$-quasi $1$-Lipschitz functions $f_i:X_i\to Y_i$, $i=1,2$, between compact metric spaces give persistent homomorphisms $f_i:=(f_i)_*: \H_d(\VR\bu(X_i)) \to \H_d(\VR\bu(Y_i))$, 
whose induced kernel, image and cokernel persistence modules are 
$2\epsilon+\delta$-interleaved, provided the functions lay at distance $d_{GH}(f_1,f_2)\leq \epsilon$, so that: 
\[  \max\big\{d_I(\Ker(f_1),\Ker(f_2)),~ d_I(\Img(f_1),\Img(f_1)),~d_I(\Coker(f_1),\Coker(f_2)) ~\big\}\leq~ 2\cdot d_{GH}(f_1,f_2)+\delta \]
\end{cor}

\begin{thm}\label{thm main 2 - stability}
The persistent cost associated to $1$-Lipschitz functions between compact metric spaces $f,g:X\to Y$ is stable with respect to the Gromov-Hausdorff distance between maps, more precisely, we have: 
\[ |C(f)-C(g)| \leq 2\cdot d_{GH}(f,g).  \]
\end{thm}

\begin{pf}
\[ \begin{array}{rl}
|C(f)-C(g)| & =\big| \max\{d_I(\Ker(f),\Delta), d_I(\Coker(f),\Delta)\} - \max\{d_I(\Ker(g),\Delta), d_I(\Coker(g),\Delta)\} \big|  \\[8pt]
& = \big|d_I(\Ker\oplus\Coker(f),\Delta) - d_I(\Ker\oplus\Coker(g),\Delta)\big| \\[8pt]
& \leq d_I\big(\Ker\oplus\Coker(f),\Ker\oplus\Coker(g)\big) \\[8pt]
& \leq max\{~d_I(\Ker(f),\Ker(g)) ~,~ d_I(\Coker(f),\Coker(g)) ~\} \\[8pt]
& \leq 2\cdot d_{GH}(f,g).
\end{array} \]
    We have sequentially applied the triangular inequality and previous Lemma \ref{lem bounding (co)kernel, images} on the final step.

\end{pf}

\subsection{A metric example}

The interleaving distance between persistence modules can be characterized via persistence homomorphisms: two persistence modules $V_{\bullet}$ and $W_{\bullet}$ are $\epsilon$-interleaved if and only if there exists an $\epsilon$-degree persistence homomorphism $F: V_{\bullet} \to W_{\bullet+\epsilon}$ with persistent cost $C(F) = 2\epsilon$ (\cite{bauer2015induced}, Remark 6).

A natural question is whether such optimal homomorphisms can be realized geometrically, i.e., as maps induced by 1-Lipschitz or \(\epsilon\)-quasi 1-Lipschitz functions between the underlying metric spaces. The following example shows that this is not always possible, i.e., the persistent cost of any 1-Lipschitz induced map may be larger than the algebraically optimal value.

\bigskip
\noindent\textbf{Example.}
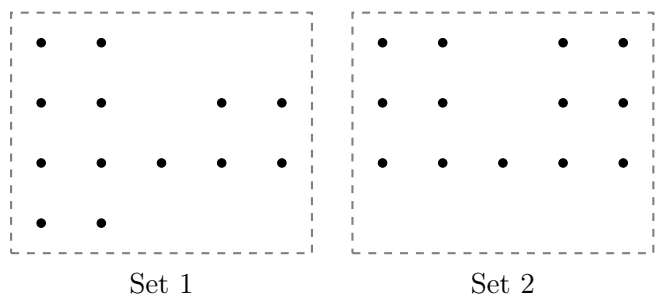
\begin{figure}[ht] 
\centering
\begin{tikzpicture}[scale=0.7]

\filldraw (0,0) circle (2pt);
\filldraw (0,-1) circle (2pt);
\filldraw (1,0) circle (2pt);
\filldraw (1,-1) circle (2pt);
\filldraw (0,1) circle (2pt);
\filldraw (0,-2) circle (2pt);
\filldraw (1,1) circle (2pt);
\filldraw (1,-2) circle (2pt);
\filldraw (2,-1) circle (2pt);
\filldraw (3,0) circle (2pt);
\filldraw (3,-1) circle (2pt);
\filldraw (4,0) circle (2pt);
\filldraw (4,-1) circle (2pt);
  
    \draw[thick, dashed, gray] (-0.5,-2.5) rectangle (4.5,1.5);
    \node at (2,-3) {Set 1};
\end{tikzpicture}
\quad
\begin{tikzpicture}[scale=0.7] 
  
\filldraw (0,-1) circle (2pt);
\filldraw (1,0) circle (2pt);
\filldraw (-1,0) circle (2pt);
\filldraw (2,0) circle (2pt);
\filldraw (-2,0) circle (2pt);
\filldraw (2,1) circle (2pt);
\filldraw (-2,1) circle (2pt);
\filldraw (2,-1) circle (2pt);
\filldraw (-2,-1) circle (2pt);
\filldraw (1,1) circle (2pt);
\filldraw (-1,1) circle (2pt);
\filldraw (1,-1) circle (2pt);
\filldraw (-1,-1) circle (2pt);
    \node at (0,-3) {Set 2};
    \draw[thick, dashed, gray] (-2.5,-2.5) rectangle (2.5,1.5);
\end{tikzpicture}
\label{fig 2sets}
\caption{Two sets 
with identical barcodes which do not admit a $1$-Lipschitz map with small persistent cost.}
\end{figure}
Consider the finite sets presented in Figure \ref{fig 2sets} and notice that they have the same diagram in persistent homology of arbitrary dimension. 
In fact, the \textit{dust phase} has the same 13 bars in $H_0$ that collapse into only one at scale $\epsilon=1$. 
On the $H_1$ level: four different classes are born at $\epsilon=1$ and die by $\epsilon=\sqrt{2}$. Higher dimensional homology is trivial. 

However, as an exercise on non-expanding functions, one can see that it is impossible to produce such a map with trivial or at even small persistent cost in dimensions $d=0$, nor in dimension $d=1$. In fact, any $1$-Lipschitz function defined in either direction between Set 1 and Set 2 will necessarily leave one of the \enquote{corner} points unmatched and hence introduce a defect in both $H_0$ and $H_1$ with positive persistence lifespan. 

Likewise, we cannot define an $\epsilon$-quasi $1$-Lipschitz map, with arbitrarily small $\epsilon>0$.

\bibliographystyle{alpha}
\bibliography{biblio}

\end{document}